\documentclass[reqno,11pt]{amsart}
\usepackage{amssymb,amsmath}
\numberwithin{equation}{section}
\newtheorem{theorem}[equation]{Theorem}
\newtheorem{lemma}[equation]{Lemma}

\newtheorem{corollary}[equation]{Corollary}

\theoremstyle{definition}
\newtheorem{example}[equation]{Example}
\newtheorem*{remark}{Remark}

\def\L{\mathbf{L}}
\def\P{\mathbf{P}}
\def\Q{Q}

\title[Convolution and inverse relations for partial Bell polynomials]{Some convolution identities and an inverse relation involving partial Bell polynomials}
\author{Daniel Birmajer}
\address{Department of Mathematics\\ Nazareth College\\ 4245 East Ave.\\ Rochester, NY 14618}
\author{Juan B. Gil}
\address{Penn State Altoona\\ 3000 Ivyside Park\\ Altoona, PA 16601}
\author{Michael D. Weiner}

\keywords{Inverse relations, partial Bell polynomials, convolution identities}

\begin{document}

\begin{abstract}
We prove an inverse relation and a family of convolution formulas involving partial Bell polynomials. Known and some presumably new combinatorial identities of convolution type are discussed. Our approach relies on an interesting multinomial formula for the binomial coefficients. The inverse relation is deduced from a parametrization of suitable identities that facilitate dealing with compositions of Bell polynomials.
\end{abstract}

\maketitle

\section{Introduction}
The main goal of this paper is to obtain the inverse relation (Theorem~\ref{thm:invRel}):
\begin{equation}\label{InverseRelation}
\begin{aligned}
y_n =& \sum_{k=1}^n \binom{an+bk}{k-1}(k-1)!B_{n,k}(x),\\ 
x_n =& \sum_{k=1}^n \frac{an+bk}{an+b} \binom{-an-b}{k-1}(k-1)!B_{n,k}(y),
\end{aligned}
\end{equation}
where $a$ and $b$ are integers (not both equal to 0), $x$ and $y$ stand for $x=(x_1,x_2,\dotsc)$ and $y=(y_1,y_2,\dots)$, and $B_{n,k}(z)$ denotes the $(n,k)$th partial Bell polynomial in the variables $z_1,z_2,\dots,z_{n-k+1}$. There is a vast literature about Bell polynomials and their applications, see e.g. \cite{Bell,Charalambides,Comtet,WW09}. Recall that
\begin{equation*}
B_{n,k}(z) =\sum_{i\in\pi(n,k)}\frac{n!}{i_1!i_2!\cdots}\left(\frac{z_1}{1!}\right)^{i_1}\left(\frac{z_2}{2!}\right)^{i_2}\cdots,
\end{equation*}
where $\pi(n,k)$ is the set of all sequences $i=(i_1,i_2,\dots)$ of nonnegative integers such that
\begin{equation*}
i_1+i_2+\dots=k \;\text{ and }\; i_1+2i_2+3i_3+\dots=n.
\end{equation*}

The above relation generalizes similar inverse relations available in the literature. For instance, with $b=1$ we recover a result by Comtet \cite[Theorem~F, p.~151]{Comtet}, which he obtained by means of the inversion formula of Lagrange. Moreover, for $a=0$ and $b=1$, the inverse pair \eqref{InverseRelation} fits into the Fa{\`a} di Bruno relations given by Chou, Hsu, and Shiue \cite{CHS06}.

To achieve \eqref{InverseRelation} we develop an alternative approach. We consider a parametrization of suitable identities (Theorem~\ref{lambdaTheorem}) that allow us to deal with nested compositions of partial Bell polynomials. This approach seems novel and relies on interesting (presumably new) convolution formulas for partial Bell polynomials, see Section~\ref{sec:BellPolynomials}.

In particular, if $\alpha(\ell,m)$ is a polynomial in $\ell$ and $m$ of degree at most one, then for any sequence $x=(x_1,x_2,\dots)$ we get (Corollary~\ref{cor:Bellconvolution}):%
\footnote{Throughout this paper, the choice of parameters is restricted to those for which the expressions are defined.}
\begin{equation*}
\sum_{\ell=0}^{k} \sum_{m=\ell}^n \frac{\tau\binom{\alpha(\ell,m)}{k-\ell}\binom{\tau-\alpha(\ell,m)}{\ell}\binom{n}{m}}{\alpha(\ell,m)\big(\tau-\alpha(\ell,m)\big)\binom{k}{\ell}} B_{m,\ell}(x) B_{n-m,k-\ell}(x) 
= \tfrac{\tau-\alpha(0,0)+\alpha(k,n)}{\alpha(k,n)(\tau-\alpha(0,0))}\binom{\tau}{k} B_{n,k}(x).
\end{equation*}
A general convolution identity is presented in Theorem~\ref{thm:BellConvolution}, and cases of particular interest are given in Corollary~\ref{cor:basicConvolution}. For the special case when $\alpha(\ell,m)=r$ is a positive integer $r\le k$, we recover the identity
\begin{equation*}
\binom{k}{r} B_{n,k}(x) = \sum_{m=k-r}^{n-r} \binom{n}{m} B_{m,k-r}(x) B_{n-m,r}(x),
\end{equation*}
recently discovered by Cvijovi{\'c} \cite[Eqn. (1.4)]{Cvijovic}. Identities involving Bell polynomials have of course the usual direct consequences for the Stirling numbers of first and second kind.

\medskip
This paper is essentially self-contained. We start with a multi-variable version of the well known identity $\sum_{j=0}^n (-1)^j\binom{n}{j}P(j) =0$ for any polynomial $P(x)$ of degree less than $n$, and use it to prove Theorem~\ref{thm:binomial}. This theorem provides core identities for the convolution formulas of Bell polynomials given in Section~\ref{sec:BellPolynomials}. In Section~\ref{sec:InverseRelation}, we prove the inverse relation \eqref{InverseRelation} and discuss some special cases. 

It is worth noting that the results of Section~\ref{sec:convolutions} can be viewed as a template for a variety of combinatorial identities of convolution type. For instance, special choices of the parameters in equations \eqref{eq:th1a}, \eqref{eq:th1b}, and \eqref{eq:th1c} give us the famous Hagen--Rothe identities.

\section{Convolution identities}\label{sec:convolutions}

Given $n\in\mathbb{N}$ it is known that, for any polynomial $P(x)$ of degree less than $n$,
\begin{equation*}
 \sum_{j=0}^n (-1)^j\binom{n}{j}P(j) =0.
\end{equation*}

For several variables, we have:
\begin{lemma}\label{lem:identity0}
Given nonnegative integers $v_1, \dots, v_d$, 
\begin{equation*}
 \sum_{i_1,\dots,i_d\ge 0} (-1)^{i_1+\dots+i_d}\binom{v_1}{i_1}\cdots\binom{v_d}{i_d} P(i_1,\dots,i_d)=0.
\end{equation*}
for any polynomial $P(x_1,\dots,x_d)$ of degree less than $k=v_1+\dots+v_d$.
\end{lemma}
This lemma, which can be proved by reduction to the one-variable case, is the base of our theorem below.

\bigskip
For a given $v=(v_1,v_2,\dots,v_d)$ with $0\le v_i\in\mathbb{Z}$, and such that $v\not=0$, define
\begin{equation}\label{eq:binomProducts}
 W_{m,\ell}(v)=\sum_{i\in \pi_d(m,\ell)}
 \!\! \binom{v_1}{i_1} \cdots \binom{v_d}{i_d},
\end{equation}
where $\pi_d(m,\ell)$ is the set of all $i=(i_1,i_2,\dots,i_d)\in\mathbb{N}_0^d$ such that
\begin{equation*}
i_1+\dots+i_d=\ell \;\text{ and }\; i_1+2i_2+\dots+di_d=m.
\end{equation*}
For those values of $(m,\ell)$ for which $\pi_d(m,\ell)$ is empty, we set $W_{m,\ell}(v)=0$. In particular, $W_{m,0}(v)=0$ unless $m=0$. 

\begin{theorem} \label{thm:binomial}
Let $v=(v_1,v_2,\dots,v_d)$ be any finite sequence of nonnegative integers with $v_1+\cdots+v_d=k>0$ and $v_1+2v_2+\cdots+dv_d=n$. Let $\alpha(\ell,m)$ be a polynomial in $\ell$ and $m$ of degree at most one. For any $\tau\in\mathbb{C}$, we have
\begin{align} 
\label{eq:th1a}
\sum_{\ell=0}^{k}\sum_{m=\ell}^n \frac{\alpha(k,n)}{\alpha(\ell,m)}\frac{\binom{\alpha(\ell,m)}{k-\ell}\binom{\tau-\alpha(\ell,m)}{\ell}}{\binom{k}{\ell}}W_{m,\ell}(v) &=\binom{\tau}{k}, \text{ and} \\
\label{eq:th1b}
\sum_{\ell=0}^{k}\sum_{m=\ell}^n \frac{\alpha(0,0)}{\alpha(\ell,m)}\frac{\binom{\tau-\alpha(\ell,m)}{k-\ell}\binom{\alpha(\ell,m)}{\ell}}{\binom{k}{\ell}}W_{m,\ell}(v) &=\binom{\tau}{k}.
\end{align}
\end{theorem}
\begin{proof}
For simplicity, we will write $\alpha=\alpha(\ell,m)$. At first, we assume $\tau$ to be an integer with $0\le \tau \le k-1$. For such $\tau$ we have
\begin{align*}
\frac{1}{\alpha}\frac{\binom{\alpha}{k-\ell}\binom{\tau-\alpha}{\ell}}{\binom{k}{\ell}}
&= \frac{(-1)^\ell}{\alpha}\frac{\binom{\alpha}{k-\ell}\binom{\alpha-\tau+\ell-1}{\ell}}{\binom{k}{\ell}}\\
&= \frac{(-1)^\ell\, \tau!(k-\tau-1)!}{k!} \binom{\alpha-1}{\tau}\binom{\alpha-\tau+\ell-1}{k-\tau-1},
\end{align*}
and so
\begin{align*}
\sum_{\ell=0}^{k}\sum_{m=\ell}^n & \;\frac{1}{\alpha}\frac{\binom{\alpha}{k-\ell}\binom{\tau-\alpha}{\ell}}{\binom{k}{\ell}}W_{m,\ell}(v) \\
&= \sum_{\ell=0}^{k}\sum_{m=\ell}^n \frac{(-1)^\ell\, \tau!(k-\tau-1)!}{k!} \binom{\alpha-1}{\tau}\binom{\alpha-\tau+\ell-1}{k-\tau-1} W_{m,\ell}(v) \\
&= \frac{\tau!(k-\tau-1)!}{k!} \sum_{\ell=0}^{k}\sum_{m=\ell}^n  (-1)^\ell\binom{\alpha-1}{\tau}\binom{\alpha-\tau+\ell-1}{k-\tau-1} W_{m,\ell}(v).
\end{align*}
Since $\alpha=\alpha(\ell,m)$ is a polynomial in $\ell$ and $m$ of degree at most one, the term $\binom{\alpha-1}{\tau}\binom{\alpha-\tau+\ell-1}{k-\tau-1}$ is a polynomial in $\ell$ and $m$ of degree at most $k-1$.  Hence by Lemma~\ref{lem:identity0}, 
\begin{equation*}
\sum_{\ell=0}^{k}\sum_{m=\ell}^n  (-1)^\ell\binom{\alpha-1}{\tau}\binom{\alpha-\tau+\ell-1}{k-\tau-1} W_{m,\ell}(v)=0,
\end{equation*}
and so
\begin{equation*}
\sum_{\ell=0}^{k}\sum_{m=\ell}^n \frac{\alpha(k,n)}{\alpha(\ell,m)}\frac{\binom{\alpha(\ell,m)}{k-\ell}\binom{\tau-\alpha(\ell,m)}{\ell}}{\binom{k}{\ell}}W_{m,\ell}(v)=0 \;\text{ for } \tau=0,1,\dots, k-1.
\end{equation*}
Now, both sides of equation \eqref{eq:th1a} are polynomials in $\tau$ of degree $k$, having the same $k$ roots $0,1,\dots,k-1$. The principal coefficient of $\binom{\tau}{k}$ is $\frac{1}{k!}$, and for the left hand side, the only term invoving $\tau^k$ is obtained when $\ell=k$ (impliying $m=n$), that is $\binom{\tau-\alpha(k,n)}{k}$. Thus its principal coefficient is also $\frac{1}{k!}$. This gives \eqref{eq:th1a} for every $\tau\in\mathbb{C}$. The identity \eqref{eq:th1b} can be verified using the same argument; this is left to the reader. 
\end{proof}

\begin{corollary}\label{cor:partialfractions}
Let $\alpha(\ell,m)$ be as in Theorem~\ref{thm:binomial}. For any $v\in\pi_d(n,k)$,
\begin{equation}\label{eq:th1c}
\sum_{\ell=0}^{k}\sum_{m=\ell}^n \frac{\tau\binom{\alpha(\ell,m)}{k-\ell}\binom{\tau-\alpha(\ell,m)}{\ell}}{\alpha(\ell,m)\big(\tau-\alpha(\ell,m)\big)\binom{k}{\ell}}W_{m,\ell}(v) 
=\frac{\tau-\alpha(0,0)+\alpha(k,n)}{\alpha(k,n)\big(\tau-\alpha(0,0)\big)}\binom{\tau}{k}.
\end{equation}
\end{corollary}
\begin{proof}
Replacing $\alpha$ by $\tau-\alpha$ in \eqref{eq:th1b}, we obtain
\begin{equation*}
\sum_{\ell=0}^{k}\sum_{m=\ell}^n \frac{\tau-\alpha(0,0)}{\tau-\alpha(\ell,m)}\frac{\binom{\alpha(\ell,m)}{k-\ell}\binom{\tau-\alpha(\ell,m)}{\ell}}{\binom{k}{\ell}}W_{m,\ell}(v) =\binom{\tau}{k}.
\end{equation*}
This identity together with \eqref{eq:th1a} give the claimed formula.
\end{proof}

In the case when $\alpha(\ell,m)=\alpha(\ell)$, the left-hand side of \eqref{eq:th1b} can be written as 
\begin{equation*}
\sum_{\ell=0}^{k}\sum_{m=\ell}^n \frac{\alpha(0,0)}{\alpha(\ell,m)}\frac{\binom{\tau-\alpha(\ell,m)}{k-\ell}\binom{\alpha(\ell,m)}{\ell}}{\binom{k}{\ell}} W_{m,\ell}(v)
= \sum_{\ell=0}^{k} \frac{\alpha(0)}{\alpha(\ell)} \binom{\tau-\alpha(\ell)}{k-\ell}\binom{\alpha(\ell)}{\ell}
\end{equation*}
since $\sum_{m=\ell}^n W_{m,\ell}(v)=\binom{k}{\ell}$ for any $v\in\pi_d(n,k)$. Thus \eqref{eq:th1b} becomes
\begin{equation}\label{eq:th1bL}
\sum_{\ell=0}^{k} \frac{\alpha(0)}{\alpha(\ell)} \binom{\tau-\alpha(\ell)}{k-\ell}\binom{\alpha(\ell)}{\ell} = \binom{\tau}{k},
\end{equation}
and similarly, \eqref{eq:th1c} turns into
\begin{equation}\label{eq:th1cL}
 \sum_{\ell=0}^{k} \frac{\tau\binom{\alpha(\ell)}{k-\ell}\binom{\tau-\alpha(\ell)}{\ell}}{\alpha(\ell)(\tau-\alpha(\ell))} = \frac{\tau-\alpha(0)+\alpha(k)}{\alpha(k)\big(\tau-\alpha(0)\big)}\binom{\tau}{k}.
\end{equation}

\begin{example}
With appropriate choices of $\tau$ and $\alpha(\ell)$, the above equations give the well-known Hagen--Rothe identities. For instance, with $\tau=x+y+kz$ and $\alpha(\ell)=y+(k-\ell)z$, the identity \eqref{eq:th1cL} gives 
\begin{equation*}
 \sum_{\ell=0}^{k} \frac{x}{x+\ell z} \binom{x+\ell z}{\ell} \frac{y}{y+(k-\ell)z} \binom{y+(k-\ell)z}{k-\ell} = \frac{x+y}{x+y+kz}\binom{x+y+kz}{k},
\end{equation*}
and with $\tau=x+y+kz$ and $\alpha(\ell)=x+\ell z$, identity \eqref{eq:th1bL} leads to
\begin{equation*}
 \sum_{\ell=0}^{k} \frac{x}{x+\ell z} \binom{x+\ell z}{\ell} \binom{y+(k-\ell)z}{k-\ell} = \binom{x+y+kz}{k}.
\end{equation*}
The special case when $z=0$ is known as the Chu--Vandermonde identity:
\begin{equation*}
\sum_{\ell=0}^{k} \binom{x}{\ell}\binom{y}{k-\ell}=\binom{x+y}{k}.
\end{equation*}
\end{example}

\begin{example}
For $\tau=-1$, identity \eqref{eq:th1b} gives
\begin{equation*}
\sum_{\ell=0}^{k}\sum_{m=\ell}^n \frac{\alpha(0,0)}{\alpha(\ell,m)}\frac{\binom{-1-\alpha(\ell,m)}{k-\ell}\binom{\alpha(\ell,m)}{\ell}}{\binom{k}{\ell}}W_{m,\ell}(v) =(-1)^{k}.
\end{equation*}
Now, since
\begin{equation*}
 \binom{-1-\alpha(\ell,m)}{k-\ell}\binom{\alpha(\ell,m)}{\ell} 
 = (-1)^{k-\ell} \binom{\alpha(\ell,m)+k-\ell}{k}\binom{k}{\ell},
\end{equation*}
we then get
\begin{equation}\label{eq:negativeOne}
\sum_{\ell=0}^{k}\sum_{m=\ell}^n (-1)^{\ell} \frac{\alpha(0,0)}{\alpha(\ell,m)} \binom{\alpha(\ell,m)+k-\ell}{k} W_{m,\ell}(v) = 1.
\end{equation}
Choosing $\alpha(\ell,m)=z-k+\ell$ with $z>k$, and since $\sum_{m=\ell}^n W_{m,\ell}(v)=\binom{k}{\ell}$, we get
\begin{equation*}
1 = \sum_{\ell=0}^{k} (-1)^{\ell} \frac{z-k}{z-k+\ell}\binom{z}{k}\binom{k}{\ell}
=\binom{z}{k} \sum_{\ell=0}^{k} (-1)^{\ell-1} \binom{k}{\ell} \frac{\ell}{z-k+\ell}.
\end{equation*}
In other words, we recover the known formula
\begin{equation*}
\frac{1}{\binom{z}{k}} = \sum_{\ell=1}^{k} (-1)^{\ell-1} \binom{k}{\ell} \frac{\ell}{z-k+\ell}.
\end{equation*}
\end{example}

\begin{example}
Let $0\le z\le k$ and $\gamma\in\mathbb{N}$. Choose $v=(v_1,v_2,\dots)$ such that $v_{\gamma}=z$, $v_{\gamma+1}=k-z$, and $v_j=0$ for every other $j$. Then
\begin{equation*}
W_{m,\ell}(v)=\binom{z}{(\gamma+1)\ell-m} \binom{k-z}{m-\gamma\ell}
\end{equation*}
and \eqref{eq:negativeOne} turns into
\begin{equation*}
\sum_{\ell=0}^{k}\sum_{m=\gamma\ell}^{(\gamma+1)\ell} (-1)^{\ell} \frac{\alpha(0,0)}{\alpha(\ell,m)} \binom{\alpha(\ell,m)+k-\ell}{k} \binom{z}{(\gamma+1)\ell-m} \binom{k-z}{m-\gamma\ell} = 1.
\end{equation*}
Other identities can be obtained by choosing a different $v$. 
\end{example}

\bigskip
We finish this section with a simple and straightforward abstraction of Theorem~\ref{thm:binomial}.
\begin{theorem}\label{thm:generalBinomial}
Let $v=(v_1,v_2,\dots,v_d)$ be any finite sequence of nonnegative integers with $v_1+\cdots+v_d=k>0$ and $v_1+2v_2+\cdots+dv_d=n$. Let $p_{m,\ell}(\tau)$ be polynomials in $\tau$ of degree at most $\ell$, and such that for every $\tau_0\in\{0,1,\dots,k-1\}$, the function $(m,\ell)\mapsto p_{m,\ell}(\tau_0)$ is a polynomial in $m$ and $\ell$ of degree at most $k-1$. If $p_{n,k}(\tau)$ has degree $k$ in $\tau$, then
\begin{equation*} 
\sum_{\ell=0}^{k}\sum_{m=\ell}^n \frac{(-1)^\ell}{k!}p_{m,\ell}(\tau) W_{m,\ell}(v) =\gamma_k\binom{\tau}{k}, 
\end{equation*}
where $\gamma_k$ is the coefficient of $\tau^k$ in $(-1)^k p_{n,k}(\tau)$.
\end{theorem}

For example, in equation \eqref{eq:th1a} we have 
\begin{gather*}
p_{m,\ell}(\tau) = (-1)^\ell k! \frac{\alpha(k,n)}{\alpha(\ell,m)}\frac{\binom{\alpha(\ell,m)}{k-\ell}\binom{\tau-\alpha(\ell,m)}{\ell}}{\binom{k}{\ell}},
\intertext{which for $\tau_0\in \{0,1,\dots,k-1\}$ can be written (cf. proof of Thm.~\ref{thm:binomial}) as} 
p_{m,\ell}(\tau_0) = \tau!(k-\tau_0-1)!\, \alpha(k,n)\binom{\alpha(\ell,m)-1}{\tau_0}\binom{\alpha(\ell,m)-\tau_0+\ell-1}{k-\tau_0-1}.
\end{gather*}
Recall that $\alpha(\ell,m)$ is a polynomial in $\ell$ and $m$ of degree at most one, thus $(m,\ell)\mapsto p_{m,\ell}(\tau_0)$ is a polynomial in $\ell$ and $m$ of degree at most $k-1$.

\section{Identities for partial Bell polynomials}
\label{sec:BellPolynomials}

In this section we derive some convolution formulas for the partial Bell polynomials:
\begin{align*}
B_{n,k}(x_1,x_2,\dots) =\sum_{i\in\pi(n,k)}\frac{n!}{i_1!i_2!\cdots}\left(\frac{x_1}{1!}\right)^{i_1}\left(\frac{x_2}{2!}\right)^{i_2}\dots,
\end{align*}
where $\pi(n,k)$ is the set of all sequences $i=(i_1,i_2,\dots)$ of nonnegative integers such that
\begin{equation*}
i_1+i_2+\dots=k \;\text{ and }\; i_1+2i_2+3i_3+\dots=n.
\end{equation*}
\begin{lemma}\label{lem:Bellproduct} 
For any $x=(x_1,x_2,\dots)$ and integers $0\le\ell\le k$ and $0\le m\le n$, we have
\begin{equation*}
\binom{n}{m} B_{m,\ell}(x) B_{n-m,k-\ell}(x)= \!\!\sum_{v\in\pi(n,k)}\frac{n!}{v_1!v_2!\dotsb}W_{m,\ell}(v)\left(\frac{x_1}{1!}\right)^{v_1}\left(\frac{x_2}{2!}\right)^{v_2}\cdots, 
\end{equation*}
where $W_{m,\ell}(v)$ is defined as in \eqref{eq:binomProducts}.
\end{lemma}
\begin{proof}
This is a straightforward consequence of the definitions. For simplicity, we let $z_j=\frac{x_j}{j!}$ and use the convenient multinomial notation
\begin{equation*}
 \binom{n}{v}=\frac{n!}{v_1!v_2!\dotsb}, \quad z^v = z_1^{v_1}z_2^{v_2}\dots, \text{ and so on.}
\end{equation*}
Thus
\begin{align*}
\binom{n}{m} B_{m,\ell}(x) & B_{n-m,k-\ell}(x)  \\
&= \binom{n}{m} \left(\sum_{i\in\pi(m,\ell)}\binom{m}{i}z^i \right)\left(\sum_{j\in\pi(n-m,k-\ell)}\binom{n-m}{j}z^j \right)\\
&=\binom{n}{m} \sum_{v\in\pi(n,k)}\sum_{i\in\pi(m,\ell)} \binom{m}{i}\binom{n-m}{v-i}z^{v} 
\\[-3ex] \intertext{}
&= \sum_{v\in\pi(n,k)} \frac{n!}{v_1!v_2!\cdots} \left(\sum_{i\in\pi(m,\ell)}\binom{v_1}{i_1}\binom{v_2}{i_2}\cdots\right) z^v \\
&= \sum_{v\in\pi(n,k)} \frac{n!}{v_1!v_2!\cdots} W_{m,\ell}(v) \left(\frac{x_1}{1!}\right)^{v_1}\left(\frac{x_2}{2!}\right)^{v_2}\cdots
\end{align*}
\end{proof}
 
As a consequence of Theorem~\ref{thm:generalBinomial} and the previous lemma, we obtain the following convolution formula:

\begin{theorem}\label{thm:BellConvolution} 
Let $p_{m,\ell}(\tau)$ be polynomials in $\tau$ with $\deg(p_{m,\ell})\le\ell$, and such that for every $\tau_0\in\{0,1,\dots,k-1\}$, the function $(m,\ell)\mapsto p_{m,\ell}(\tau_0)$ is a polynomial in $m$ and $\ell$ of degree at most $k-1$. If $p_{n,k}(\tau)$ has degree $k$ in $\tau$, then for any sequence $x=(x_1,x_2,\dots)$,
\begin{equation*} 
\sum_{\ell=0}^{k}\sum_{m=\ell}^n \frac{(-1)^\ell}{k!}p_{m,\ell}(\tau)\binom{n}{m} B_{m,\ell}(x) B_{n-m,k-\ell}(x) = \gamma_k\binom{\tau}{k} B_{n,k}(x), 
\end{equation*}
where $\gamma_k$ is the coefficient of $\tau^k$ in $(-1)^k p_{n,k}(\tau)$.
\end{theorem}
\begin{proof}
By Lemma~\ref{lem:Bellproduct}, we have
\begin{align*}
\sum_{\ell=0}^{k}\sum_{m=\ell}^n &\, \frac{(-1)^\ell}{k!\gamma_k} p_{m,\ell}(\tau)\binom{n}{m} B_{m,\ell}(x) B_{n-m,k-\ell}(x) \\
&= \sum_{\ell=0}^{k}\sum_{m=\ell}^n \frac{(-1)^\ell}{k!\gamma_k}p_{m,\ell}(\tau) \sum_{v\in\pi(n,k)}\frac{n!}{v_1!v_2!\dotsb}W_{m,\ell}(v)\left(\frac{x_1}{1!}\right)^{v_1}\!\left(\frac{x_2}{2!}\right)^{v_2}\cdots \\
&= \! \sum_{v\in\pi(n,k)} \frac{n!}{v_1!v_2!\cdots} \left(\sum_{\ell=0}^{k}\sum_{m=\ell}^n \frac{(-1)^\ell}{k!\gamma_k}p_{m,\ell}(\tau) W_{m,\ell}(v)\right) \left(\frac{x_1}{1!}\right)^{v_1}\!\left(\frac{x_2}{2!}\right)^{v_2}\cdots
\intertext{which by Theorem~\ref{thm:generalBinomial} becomes}
&= \binom{\tau}{k} \sum_{v\in\pi(n,k)} \frac{n!}{v_1!v_2!\cdots} \left(\frac{x_1}{1!}\right)^{v_1}\!\left(\frac{x_2}{2!}\right)^{v_2}\cdots = \binom{\tau}{k} B_{n,k}(x). 
\end{align*} 
\end{proof}

For the special cases in Theorem~\ref{thm:binomial}, we obtain:
\begin{corollary}\label{cor:basicConvolution}
Let $\alpha(\ell,m)$ be a polynomial in $\ell$ and $m$ of degree at most one. For any sequence $x=(x_1,x_2,\dots)$ and any $\tau\in\mathbb{C}$, we have
\begin{align*} 
\sum_{\ell=0}^{k}\sum_{m=\ell}^n \frac{\alpha(k,n)}{\alpha(\ell,m)}\frac{\binom{\alpha(\ell,m)}{k-\ell}\binom{\tau-\alpha(\ell,m)}{\ell}\binom{n}{m}}{\binom{k}{\ell}} B_{m,\ell}(x) B_{n-m,k-\ell}(x) &= \binom{\tau}{k} B_{n,k}(x), \\
\sum_{\ell=0}^{k}\sum_{m=\ell}^n \frac{\alpha(0,0)}{\alpha(\ell,m)}\frac{\binom{\tau-\alpha(\ell,m)}{k-\ell}\binom{\alpha(\ell,m)}{\ell}\binom{n}{m}}{\binom{k}{\ell}} B_{m,\ell}(x) B_{n-m,k-\ell}(x) &= \binom{\tau}{k} B_{n,k}(x).
\end{align*}
\end{corollary}

And from Corollary~\ref{cor:partialfractions}, we deduce:
\begin{corollary}\label{cor:Bellconvolution}
For any sequence $x=(x_1,x_2,\dots)$, we have
\begin{align*}
\sum_{\ell=0}^{k}\sum_{m=\ell}^n \frac{\tau\binom{\alpha(\ell,m)}{k-\ell}\binom{\tau-\alpha(\ell,m)}{\ell}\binom{n}{m}}{\alpha(\ell,m)\big(\tau-\alpha(\ell,m)\big)\binom{k}{\ell}} B_{m,\ell}(x) B_{n-m,k-\ell}(x) 
=\tfrac{\tau-\alpha(0,0)+\alpha(k,n)}{\alpha(k,n)(\tau-\alpha(0,0))}\binom{\tau}{k} B_{n,k}(x).
\end{align*}
\end{corollary}

\begin{example}\label{ex:SimpleConvolution}
Let $r$ be an integer with $0< r \le k$.  In the special case when $\alpha(\ell,m)=r$ and $\tau=k$, the first identity in Corollary~\ref{cor:basicConvolution} gives
\begin{equation*} 
B_{n,k}(x) = \sum_{\ell=0}^{k}\sum_{m=\ell}^n \frac{\binom{r}{k-\ell}\binom{k-r}{\ell}\binom{n}{m}}{\binom{k}{\ell}} B_{m,\ell}(x) B_{n-m,k-\ell}(x), 
\end{equation*}
and since $\binom{r}{k-\ell}\binom{k-r}{\ell}=0$ unless $\ell=k-r$, we arrive at
\begin{equation}\label{eq:alphaConstant} 
\binom{k}{r} B_{n,k}(x) = \sum_{m=k-r}^{n-r} \binom{n}{m} B_{m,k-r}(x) B_{n-m,r}(x).
\end{equation}
Observe that when $r=1$, we get the basic recurrence formula
\begin{equation}\label{eq:1stepRecursion}
B_{n,k}(x)=\frac{1}{k}\sum_{m=k-1}^{n-1}\binom{n}{m}x_{n-m}B_{m,k-1}(x),
\end{equation}
see Comtet \cite{Comtet} (relation [3k], p.~136). In particular, since the Stirling numbers of second kind satisfy $S(n,k)=B_{n,k}(1,1,\dots)$,  formula \eqref{eq:alphaConstant} gives the known recurrence 
\begin{equation*} 
\binom{k}{r} S(n,k) = \sum_{m=k-r}^{n-r} \binom{n}{m} S(m,k-r) S(n-m,r).
\end{equation*}
The same recurrence holds for the unsigned Stirling numbers of first kind.

The identity \eqref{eq:alphaConstant} can also be found in a recent paper by Cvijovi{\'c} \cite{Cvijovic}. There the author gives furthermore a recurrence relation for $B_{n,k}$ (see \cite[equation (1.3)]{Cvijovic}) which, combined with \eqref{eq:1stepRecursion}, leads to the interesting identity
\begin{equation}\label{eq:zerosum}
\sum_{m=k-1}^{n-1} \left[\frac{1}{k}\binom{n}{m}-\binom{n-1}{m}\right] x_{n-m}B_{m,k-1}(x)=0.
\end{equation}
\end{example}

\section{Inverse relations}
\label{sec:InverseRelation}

For $n\in\mathbb{N}$, $b\in\mathbb{Z}$, $\lambda\in\mathbb{C}$, and a sequence $z=(z_1, z_2, \dots)$, we define
\begin{equation*}
\Q_{n,b}(\lambda,z) = \sum_{k=1}^{n} \binom{\lambda+bk}{k-1}(k-1)!B_{n,k}(z).
\end{equation*}
These functions appear naturally in certain compositions of formal power series. For instance, if $Z(t)=1+\sum_{n\ge 1} z_n \frac{t^n}{n!}$, then 
\begin{equation*}
 \log(Z(t)) = \sum_{n\ge 1} \L_n \frac{t^n}{n!}
\end{equation*}
with $\L_n=\sum_{k=1}^{n} (-1)^{k-1}(k-1)!B_{n,k}(z)$ (logarithmic polynomials), see Section~3.5 in the book by Comtet \cite{Comtet}. Moreover, for any complex number $r$,
\begin{equation*}
 Z(t)^r = 1+\sum_{n\ge 1} \P^{(r)}_n \frac{t^n}{n!}
\end{equation*}
with $\P^{(r)}_n=\sum_{k=1}^{n} (r)_k B_{n,k}(z)$ (potential polynomials). Now, since $(-1)^{k-1}=\binom{-1}{k-1}$ and since $(r)_k=r\binom{r-1}{k-1}(k-1)!$, we can write
\begin{equation}\label{eq:logpotPolynomials}
 \L_n(z) = Q_{n,0}(-1,z) \;\text{ and }\; \P^{(r)}_n(z)=r \Q_{n,0}(r-1,z).
\end{equation}

\begin{lemma}\label{prop:ac}
For $\lambda\in\mathbb{N}$ we have
\begin{equation*}
\Q_{n,0}(\lambda,z) = z_n+\sum_{i=1}^{\lambda} \frac{i}{\lambda+1}\sum_{m=1}^{n-1}\binom{n}{m}z_{n-m}\,\Q_{m,0}(i-1,z)
\end{equation*}
\end{lemma}
\begin{proof}
Using the identity \eqref{eq:1stepRecursion}, and since $B_{n,1}(z)=z_n$, we have
\begin{align*}
\Q_{n,0}(\lambda,z) - z_n &= 
\sum_{k=2}^{n}\binom{\lambda}{k-1}(k-1)!\biggl(\frac{1}{k}\sum_{m=k-1}^{n-1}\binom{n}{m}z_{n-m}B_{m,k-1}(z)\biggr) \\ 
&=\sum_{k=1}^{n-1}\sum_{m=k}^{n-1}\binom{\lambda}{k}\frac{k!}{k+1}\binom{n}{m}z_{n-m}B_{m,k}(z) \\
&=\sum_{m=1}^{n-1}\binom{n}{m}z_{n-m}\sum_{k=1}^{m}\frac{k!}{\lambda+1}\binom{\lambda+1}{k+1}B_{m,k}(z)\\
&=\sum_{m=1}^{n-1}\binom{n}{m}z_{n-m}\sum_{k=1}^{m}\frac{k!}{\lambda+1}\bigg(\sum_{i=1}^{\lambda}\binom{i}{k}\bigg)B_{m,k}(z) \\
&=\sum_{i=1}^{\lambda}\frac{i}{\lambda+1}\sum_{m=1}^{n-1}\binom{n}{m}z_{n-m}\sum_{k=1}^{m}\binom{i-1}{k-1}(k-1)!B_{m,k}(z).
\end{align*}
At last, replace the interior sum over $k$ by $\Q_{m,0}(i-1,z)$ and solve for $\Q_{n,0}(\lambda,z)$.
\end{proof}

\bigskip
The next lemma is straightforward.
\begin{lemma}\label{Qproduct}
Given any sequence $z=(z_1, z_2, \dots)$, the following product formula holds:
\begin{align*}
\Q_{n_1,b_1}(\lambda_1,z)  &\; \Q_{n_2,b_2}(\lambda_2,z) \\
&= \sum_{k=2}^{n_1+n_2}\sum_{\ell=1}^{n_2}\frac{k! \binom{\lambda_1+b_1(k-\ell)+1}{k-\ell}\binom{\lambda_2+b_2\ell+1}{\ell}}{(\lambda_1+b_1(k-\ell)+1) (\lambda_2+b_2\ell+1)\binom{k}{\ell}}B_{n_1,k-\ell}(z)B_{n_2,\ell}(z).
\end{align*}
\end{lemma}

\bigskip
\begin{theorem}\label{lambdaTheorem} 
Let $a,b\in\mathbb{Z}$. Given $x=(x_1,x_2,\dotsc)$, define $y=(y_1,y_2,\dots)$ by
\begin{equation*}
y_n = \Q_{n,b}(an,x) = \sum_{k=1}^n \binom{an+bk}{k-1}(k-1)!B_{n,k}(x)
\end{equation*}
for every $n\in\mathbb{N}$. Then, for any $\lambda\in\mathbb{C}$, we have
\begin{equation}\label{eq:yx}
\sum_{k=1}^n\binom{\lambda}{k-1}(k-1)!B_{n,k}(y)=\sum_{k=1}^n\binom{\lambda+an+bk}{k-1}(k-1)!B_{n,k}(x).
\end{equation}
In other words, for any $\lambda$,
\begin{equation*}
\Q_{n,0}(\lambda,y) = \Q_{n,b}(\lambda+an,x).
\end{equation*}
\end{theorem}

\begin{proof}
At first, we assume $\lambda$ to be a positive integer. As both sides of \eqref{eq:yx} are polynomials in $\lambda$, the statement will then be valid for any $\lambda\in\mathbb{C}$.

We proceed by induction in $n$. For $n=1$,
\begin{equation*}
 \Q_{1,0}(\lambda,y)=B_{1,1}(y)=y_1 = B_{1,1}(x)=\Q_{1,b}(\lambda+a,x)  \text{ for every } \lambda.
\end{equation*}

Assume $\Q_{m,0}(\lambda,y) = \Q_{m,b}(\lambda+am,x)$ for every $1\le m<n$ and any $\lambda\in\mathbb{N}$. By Lemma~\ref{prop:ac}, 
\begin{equation*}
\Q_{n,0}(\lambda,y) = y_n+\sum_{i=1}^{\lambda} \frac{i}{\lambda+1}\sum_{m=1}^{n-1}\binom{n}{m}y_{n-m}\,\Q_{m,0}(i-1,y).
\end{equation*}
Since $m<n$, we use the induction hypothesis on $\Q_{m,0}(i-1,y)$ to obtain
\begin{align*}
\Q_{n,0}(\lambda,y) &- y_n = \sum_{i=1}^{\lambda} \frac{i}{\lambda+1}\sum_{m=1}^{n-1}\binom{n}{m}y_{n-m}\,\Q_{m,b}(i-1+am,x)\\
&= \sum_{i=1}^{\lambda} \frac{i}{\lambda+1}\sum_{m=1}^{n-1}\binom{n}{m}\Q_{n-m,b}(a(n-m),x)\,\Q_{m,b}(i-1+am,x),
\intertext{ which by Lemma~\ref{Qproduct} becomes}
&= \sum_{i=1}^{\lambda} \frac{i}{\lambda+1}\sum_{m=1}^{n-1}\binom{n}{m}\sum_{k=2}^{n}\sum_{\ell=1}^m
\frac{k! \binom{\varrho-\delta+1}{k-\ell}\binom{\delta+i}{\ell}}{(\varrho-\delta+1)(\delta+i)\binom{k}{\ell}}B_{n-m,k-\ell}(x)B_{m,\ell}(x)\\
&= \sum_{i=1}^{\lambda}\sum_{k=2}^{n} \frac{i k!}{\lambda+1}\sum_{\ell=1}^{k-1}\sum_{m=\ell}^{n-1}
\frac{\binom{\varrho-\delta+1}{k-\ell}\binom{\delta+i}{\ell}\binom{n}{m}}{(\varrho-\delta+1)(\delta+i)\binom{k}{\ell}} B_{n-m,k-\ell}(x)B_{m,\ell}(x)
\end{align*}
with $\varrho=an+bk$ and $\delta=am+b\ell$. Now, with $\alpha=\varrho-\delta+1$ and $\tau=\varrho+i+1$, the identity from Corollary~\ref{cor:Bellconvolution} becomes
\begin{equation*}
\sum_{\ell=0}^{k}\sum_{m=\ell}^n \frac{(\varrho+i+1)\binom{\varrho-\delta+1}{k-\ell}\binom{\delta+i}{\ell}\binom{n}{m}}{(\varrho-\delta+1)(\delta+i)\binom{k}{\ell}} B_{m,\ell} B_{n-m,k-\ell} 
=\frac{i+1}{i}\binom{\varrho+i+1}{k} B_{n,k}.
\end{equation*}
Hence
\begin{equation*}
\sum_{\ell=0}^{k}\sum_{m=\ell}^n \frac{\binom{\varrho-\delta+1}{k-\ell}\binom{\delta+i}{\ell}\binom{n}{m}}{(\varrho-\delta+1)(\delta+i)\binom{k}{\ell}} B_{m,\ell} B_{n-m,k-\ell} =\frac{i+1}{(\varrho+i+1)i}\binom{\varrho+i+1}{k} B_{n,k},
\end{equation*}
and so the sum from $\ell=1$ to $\ell=k-1$ (as needed above) is equal to
\begin{equation*}
\left(\frac{i+1}{(\varrho+i+1)i}\binom{\varrho+i+1}{k} - \frac{1}{\varrho+i}\binom{\varrho+i}{k} - \frac{1}{(\varrho+1)i}\binom{\varrho+1}{k}\right) B_{n,k}(x).
\end{equation*}

Therefore,
\begin{align*}
\Q_{n,0}(\lambda,y) &- y_n  \\
&= \sum_{i=1}^{\lambda}\sum_{k=2}^{n} \frac{i k!}{\lambda+1}\left( \frac{(i+1)\binom{\varrho+i+1}{k}}{(\varrho+1+i)i}-\frac{\binom{\varrho+i}{k}}{\varrho+i}-\frac{\binom{\varrho+1}{k}}{(\varrho+1) i}\right) B_{n,k}(x)\\
&= \sum_{k=2}^{n} \frac{k!}{\lambda+1}\sum_{i=1}^{\lambda}\left( \frac{(i+1)\binom{\varrho+i+1}{k}}{(\varrho+i+1)}-\frac{i\binom{\varrho+i}{k}}{\varrho+i}-\frac{\binom{\varrho+1}{k}}{\varrho+1}\right) B_{n,k}(x)\\
&= \sum_{k=2}^{n} k!\left( \frac{\binom{\varrho+\lambda+1}{k}}{(\varrho+\lambda+1)}-\frac{\binom{\varrho+1}{k}}{\varrho+1}\right) B_{n,k}(x)\\
&= \sum_{k=1}^{n} k!\left( \frac{1}{k} \binom{\varrho+\lambda}{k-1}-\frac{1}{k}\binom{\varrho}{k-1}\right) B_{n,k}(x)\\
&= \sum_{k=1}^{n} (k-1)!\left( \binom{\lambda+an+bk}{k-1}-\binom{an+bk}{k-1}\right) B_{n,k}(x)\\
&= \Q_{n,b}(\lambda+an,x) - \Q_{n,b}(an,x).
\end{align*}
Since $y_n=\Q_{n,b}(an,x)$, we get $\Q_{n,0}(\lambda,y)=\Q_{n,b}(\lambda+an,x)$ as desired.
\end{proof}

\begin{remark}
The identity \eqref{eq:yx} is also valid in the form
\begin{equation*}
\sum_{k=k_0}^n\binom{\lambda}{k-k_0}(k-1)!B_{n,k}(y)=\sum_{k=k_0}^n \binom{\lambda+an+bk}{k-k_0}(k-1)!B_{n,k}(x).
\end{equation*}
For instance, for $k_0=2$, this follows from the identity $\binom{\lambda}{k-2}=\binom{\lambda+1}{k-1}-\binom{\lambda}{k-1}$ together with the fact that \eqref{eq:yx} holds for every $\lambda$.
\end{remark}

\begin{theorem}\label{thm:invRel}
Let $a,b\in\mathbb{Z}$, not both equal to $0$. Given any sequence $x=(x_1,x_2,\dotsc)$, we consider the sequence $y=(y_1,y_2,\dots)$ defined by
\begin{equation}\label{eq:inputy}
y_n = \sum_{k=1}^n \binom{an+bk}{k-1}(k-1)!B_{n,k}(x).
\end{equation}
Then, for every $n\in\mathbb{N}$, we have
\begin{equation}\label{eq:inputx}
x_n = \sum_{k=1}^n \frac{an+bk}{an+b} \binom{-an-b}{k-1}(k-1)!B_{n,k}(y).
\end{equation}
\end{theorem}
\begin{proof}
For $k>1$ we rewrite
\begin{gather*}
\frac{an+bk}{an+b}\binom{-an-b}{k-1} = \binom{-an-b}{k-1} -b\binom{-an-b-1}{k-2}
\end{gather*}
and split the right-hand side of \eqref{eq:inputx} as
\begin{equation*}
\sum_{k=1}^n \binom{-an-b}{k-1}(k-1)!B_{n,k}(y) - b\sum_{k=2}^n \binom{-an-b-1}{k-2}(k-1)!B_{n,k}(y).
\end{equation*}
Using \eqref{eq:yx} with $\lambda=-an-b$ we obtain
\begin{align*}
\sum_{k=1}^n\binom{-an-b}{k-1}(k-1)!B_{n,k}(y) 
&=\sum_{k=1}^n\binom{b(k-1)}{k-1}(k-1)!B_{n,k}(x) \\
&=x_n + \sum_{k=2}^n \binom{b(k-1)}{k-1}(k-1)!B_{n,k}(x).
\end{align*}
Using again the identity \eqref{eq:yx} now with $\lambda=-an-b-1$, and starting at $k=2$, we get
\begin{equation*}
b\sum_{k=2}^n \binom{-an-b-1}{k-2}(k-1)!B_{n,k}(y) =\sum_{k=2}^n \binom{b(k-1)}{k-1}(k-1)!B_{n,k}(x).
\end{equation*}
The statement then follows by taking the difference.
\end{proof}

\begin{remark}
Note that equations \eqref{eq:inputy} and \eqref{eq:inputx} give an inverse relation: It is also true that for any given sequence $y=(y_1,y_2,\dots)$, the sequence $x=(x_1,x_2,\dots)$ defined by \eqref{eq:inputx} satisfies the relation \eqref{eq:inputy}.
\end{remark}

\begin{example} 
For $b=1$, Theorem~\ref{thm:invRel} gives the inverse relation
\begin{align*}
y_n &= \sum_{k=1}^n \binom{an+k}{k-1}(k-1)!B_{n,k}(x), \\
x_n &= \sum_{k=1}^n \binom{-an-2}{k-1}(k-1)!B_{n,k}(y),
\end{align*}
which can be found in \cite[Theorem~F, p. 151]{Comtet}. In particular, if $x$ is replaced by $-x$, we arrive at the symmetric relation (cf. \cite[Theorem~10]{mihou10}):
\begin{align*}
y_n &= \sum_{k=1}^n (-1)^{k} \binom{an+k}{k-1}(k-1)!B_{n,k}(x), \\
x_n &= \sum_{k=1}^n (-1)^k \binom{an+k}{k-1}(k-1)!B_{n,k}(y).
\end{align*}
\end{example}

\begin{example}[Generating function]
Consider the formal power series
\begin{equation*}
Y(t) = 1+\sum_{n\ge 1} y_n \frac{t^n}{n!}
\end{equation*}
with $y_n$ as in \eqref{eq:inputy} for any given sequence $x=(x_1,x_2,\dots)$. 

As mentioned at the beginning of the section, the logarithmic polynomials associated with $Y(t)$ are given by $\Q_{n,0}(-1,y)$, see \eqref{eq:logpotPolynomials}. By means of Theorem~\ref{lambdaTheorem}, these polynomials can then be written in terms of $x$, namely $\L_n(y) = \Q_{n,b}(-1+an,x)$, and so
\begin{equation*}
 \log\big(Y(t)\big) = \sum_{n\ge 1} \Q_{n,b}(-1+an,x)\frac{t^n}{n!}.
\end{equation*}
Similarly, for any $r$, $\P_n^{(r)}(y)=r\Q_{n,0}(r-1,y)=r\Q_{n,b}(r-1+an,x)$, thus
\begin{equation*}
 Y(t)^r = 1+\sum_{n\ge 1} r\Q_{n,b}(r-1+an,x)\frac{t^n}{n!}.
\end{equation*}
In particular, for any polynomial $F(z)=\sum_{\ell=0}^m c_\ell\, z^\ell$, we have
\begin{equation*}
 F(Y(t)) =F(1) + \sum_{n\ge 1}\left(\sum_{\ell=1}^m c_\ell\cdot \ell\Q_{n,b}(\ell-1+an,x)\right) \frac{t^n}{n!}.
\end{equation*}
\end{example}

\end{document}